\documentclass{amsart}
\usepackage[utf8]{inputenc}
\usepackage{amssymb}
\usepackage{hyperref}
\usepackage[final]{showkeys}
\usepackage[all]{xy}

\theoremstyle{definition}
\newtheorem{mydef}{Definition}[section]

\newtheorem{lemma}[mydef]{Lemma}
\newtheorem{thm}[mydef]{Theorem}
\newtheorem{theo}[mydef]{Theorem}

\newtheorem{cor}[mydef]{Corollary}

\newtheorem{propo}[mydef]{Proposition}
\newtheorem{defin}[mydef]{Definition}
\newtheorem{example}[mydef]{Example}

\newtheorem{remark}[mydef]{Remark}

\newtheorem{notation}[mydef]{Notation}
\newtheorem{fact}[mydef]{Fact}

\newcommand{\colim}{\operatorname{colim}}

\newcommand{\seq}[2]{\left( #1 \right)_{#2}}

\newbox\noforkbox \newdimen\forklinewidth
\forklinewidth=0.3pt \setbox0\hbox{$\textstyle\smile$}
\setbox1\hbox to \wd0{\hfil\vrule width \forklinewidth depth-2pt
 height 10pt \hfil}
\wd1=0 cm \setbox\noforkbox\hbox{\lower 2pt\box1\lower
2pt\box0\relax}
\def\unionstick{\mathop{\copy\noforkbox}\limits}
\newcommand{\nf}{\unionstick}

\setbox0\hbox{$\textstyle\smile$}
\setbox1\hbox to \wd0{\hfil{\sl /\/}\hfil} \setbox2\hbox to
\wd0{\hfil\vrule height 10pt depth -2pt width
               \forklinewidth\hfil}
\newbox\doesforkbox
\setbox\doesforkbox\hbox{\lower 0pt\box1 \lower
2pt\box2\lower2pt\box0\relax}

\newcommand{\ck}{\mathcal{K}}
\newcommand{\cm}{\mathcal{M}}
\newcommand{\cl}{\mathcal{L}}
\newcommand{\id}{\operatorname{id}}

\newcommand{\Rmod}{\mathbf{RMod}}

\newcommand{\Ab}{\mathbf{Ab}}

\author[Lieberman]{Michael Lieberman}
\email{qmlieberman@vutbr.cz}
\address{Institute of Mathematics, Faculty of Mechanical Engineering, Brno University of Technology, Brno, Czech Republic}

\author[Rosick\'y]{Ji\v r\'i Rosick\'y}
\email{rosicky@math.muni.cz}
\urladdr{http://www.math.muni.cz/\textasciitilde rosicky/}
\address{Department of Mathematics and Statistics, Faculty of Science, Masaryk University, Brno, Czech Republic}
\thanks{The second author is supported by the Grant agency of the Czech Republic under the grant 19-00902S}

\author[Vasey]{Sebastien Vasey}

\title{Induced and higher-dimensional stable independence}
\date{\today}

\begin{document}

\begin{abstract}
	We provide several crucial technical extensions of the theory of stable independence notions in accessible categories.  In particular, we describe circumstances under which a stable independence notion can be transferred from a subcategory to a category as a whole, and examine a number of applications to categories of groups and modules, extending results of \cite{MAFuchs}.  We prove, too, that under the hypotheses of \cite{lrv-cell}, a stable independence notion immediately yields higher-dimensional independence as in \cite{multidim-v2}. 
\end{abstract}

\subjclass{Primary: 03C45. Secondary: 03C48, 03C52, 13C60, 18C35.}
\keywords{Stable independence, accessible categories, classification theory, categories of modules}

\maketitle

\section{Introduction}

We here concern ourselves with stable independence in the context of accessible categories.  This notion has its origins in the model-theoretic concept of stable nonforking, which can be thought of on one hand as a freeness property of type extensions and, on the other, as a notion of freeness or independence for amalgams of models.  The latter perspective, taken to its logical conclusion, leads to a formulation of stable independence as a property of commutative squares in a general category, described by a family of purely category-theoretic axioms, cf. \cite{indep-categ-advances}.  This generalization is of practical value: often in mathematics we begin with a nice category of objects $\ck$, then restrict to a particular family of desirable morphisms $\cm$, obtaining a subcategory, $\ck_\cm$, which loses much of the useful structure of $\ck$.  Even if $\ck$ is locally finitely presentable (say, $\ck=\Ab$, the category of abelian groups and homomorphisms), if we take $\cm$ to be a family of monomorphisms (say, $\cm$ consists of the pure monomorphisms in $\Ab$), the category $\ck_\cm$ will no longer have the pushouts available to us in $\ck$.  A central theme of \cite{lrv-cell} is that stable independence---or, rather, stably independent squares---provide a workable alternative to the missing pushouts, sufficient for many applications.

Of greater importance, perhaps, is \cite[3.1]{lrv-cell}, which asserts, roughly speaking, that given a pair $(\ck,\cm)$, $\ck$ cocomplete, $\ck_\cm$ admits a stable independence notion just in case the family of morphisms $\cm$ is cofibrantly generated; that is, generated by pushouts, transfinite compositions and retracts from a \emph{set} (as opposed to a proper class) of morphisms.  In this case, the stable independence notion must be precisely the one given by $\cm$-effective squares, generalizing the effective unions of \cite{effective-unions}.  This provides a useful link between stable independence and, for example, (combinatorial, i.e. cofibrantly generated) cellular categories, which occur naturally in topology and homological algebra.  The translation it offers has already proved fruitful: arguing by way of stable independence, one can give a very brief proof of (a special case of) the fact that combinatorial structures are left-induced, \cite[3.11]{lrv-cell}; the original proof of this fact, in \cite{mr}, requires a great deal of heavy category-theoretic machinery, including the {\em good colimits} of Lurie (see \cite{fat-small-obj}).  On the other hand, an analysis of questions involving stable independence by way of cofibrant generation often leads to more efficient proofs, and useful new results, particularly in algebra.  For example, \cite[4.3]{lrv-cell} resolves an open question concerning the circumstances under which Ext-orthogonality classes of modules admit a stable independence notion, and, while the proof of the cofibrant generation of pure monomorphisms in locally presentable additive categories in \cite{lprv-purecofgen-v3} does not use stable independence, it would not have been evident were it not for this connection.  

We pursue further applications here, first deriving stable independence for a host of algebraic categories.  For this, we require an essential technical lemma.  A crucial part of the definition of a stable independence relation on a category $\ck$ is that it be accessible (Definition~\ref{defstab}(2)), meaning, roughly speaking, that the independent squares form an accessible subcategory of the arrow category, $\ck^2$: without this property, we say the independence relation is {\em weakly stable}.  Accessibility, which corresponds loosely to the local character of nonforking, is neither particularly natural nor easy to verify in practice.  For a {\em continuous} independence notion, on the other hand, we require only that the appropriate subcategory of $\ck^2$ be closed under directed colimits---this is far more common, as we will see in Section~\ref{secstab}.  The essential result of Section~\ref{seclift} is that, in many cases, we can infer the existence of a stable independence notion from that of a continuous weakly stable independence notion (Theorem~\ref{cofinal-thm}).  In particular, if an accessible category $\ck$ has a continuous weakly stable independence notion, and has a stable independence notion on a sufficiently nicely embedded subcategory, then in fact the independence notion on $\ck$ must be stable. Thanks to recent work on stability---in the sense of Galois types---in categories of groups and modules (e.g. \cite{univ-modules-pp-v2}, \cite{MAFuchs}), we obtain continuous weakly stable independence notions for free in, for example, Abelian $p$-groups with pure embeddings, or torsion $R$-modules with pure embeddings.  A model-theoretic argument shows that these weakly stable notions restrict to stable independence on subcategories of sufficiently saturated (that is, universal injective) objects, and Theorem~\ref{cofinal-thm} allows us to lift this stability to the categories themselves.  The results are summarised in Theorems~\ref{thmmodcats} and \ref{thmgpcats}.

Finally, we show that in the setting of \cite{lrv-cell}, the existence of a stable independence notion---essentially, a family of nice commutative squares---implies the existence of higher-dimensional stable independence: nice families of cubes, hypercubes, etc.  We note that higher-dimensional amalgams of this form have played an important role in the analysis of categoricity in $L_{\omega_1,\omega}$ (\cite{she83} and \cite{she83b}, where they were first introduced) and, more recently, in connection with the categoricity conjecture for abstract elementary classes (\cite{multidim-v2}).  In the aforementioned cases, it is a significant technical challenge to ensure that the existence of such amalgams in low dimensions can be pushed to higher dimensions (\cite{multidim-v2} needs the weak generalized continuum hypothesis, for example).  As our underlying category is locally presentable---hence has arbitrary pushouts---existence is more or less automatic. 

We assume a familiarity with accessible categories (\cite{adamek-rosicky}, \cite{makkai-pare}).  A passing acquaintance with the category-theoretic formulation of stable independence relations (\cite{indep-categ-advances}, \cite{lrv-cell}) would be very useful in providing motivation for the discussion that follows. 

We are grateful to Marcos Mazari-Armida for his detailed feedback on an early draft of this paper.  We are grateful, too, to the anonymous referee for comments that have led to several important clarifications.

\section{Lifting stable independence}\label{seclift}

While a much fuller picture can be given of the transfer of stable independence notions along adjunctions, we here concern ourselves with a very limited special case.  In particular, we show that if an accessible category $\ck$ has a \emph{$\aleph_0$-continuous} weakly stable independence notion, and has a stable independence notion on a sufficiently nice subcategory, then in fact $\ck$ has a stable independence notion.  This allows us to extend certain stability results concerning classes of modules in \cite{MAFuchs}: while that paper establishes that a number of such classes are stable in the sense of Galois types, we show that they possess stable independence notions.  This is stronger: for example, the category of $\aleph_1$-free groups and pure embeddings is Galois-stable (see \cite[5.9]{mazarmnonelem}), but does not have a stable independence notion, albeit for a rather trivial reason---this category does not have the amalgamation property.

For the sake of completeness, we briefly recall the definitions of the essential notions here.  More detailed treatments can be found in \cite{indep-categ-advances} and \cite{lrv-cell}.

\begin{defin}\label{defwstab} Let $\ck$ be a category.
	\begin{enumerate}
	\item We define an \emph{independence notion} (or \emph{independence relation}) on $\ck$ as a class $\nf$ of commutatives squares (called \emph{$\nf$-independent}, or simply \emph{independent})	such that, for any commutative diagram 
	$$
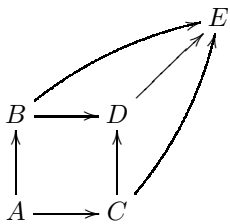
\xymatrix{ & & E\\
	B\ar[r]\ar@/^/[rru] & D\ar[ur] & \\
	A\ar[r]\ar[u] & C\ar[u]\ar@/_/[uur] & }$$
	the square spanning $A$, $B$, $C$, and $D$ is independent if and only if the square spanning $A$, $B$, $C$, and $E$ is independent.
	\item We say that an independence notion $\bf$ on $\ck$ is \emph{weakly stable} if it satisfies the following conditions:
	\begin{enumerate}
	\item Symmetry: $\nf$ is closed under reflection across the diagonal from bottom left to top right.
	\item Transitivity: $\nf$ is closed under vertical and horizontal composition of squares.
	\item Existence: Any span $B\leftarrow A\to C$ can be completed to a $\nf$-independent square.
	\item Uniqueness: Any two completions of a span $B\leftarrow A\to C$ to $\nf$-independent squares are equivalent up to amalgamation, in the sense of the diagram in (1) above.
	\end{enumerate}
	\end{enumerate}

\end{defin}

In this section, we will concern ourselves largely with the tension between \emph{weakly stable} and \emph{stable} independence, where the latter builds in a crucial accessibility condition.

\begin{defin}\label{defstab} Let $\nf$ be an independence notion on category $\ck$, satisfying the existence and transitivity properties.
	\begin{enumerate} 
	\item We denote by $\ck_\downarrow$ the subcategory of the arrow category $\ck^2$ whose objects are $\ck$-morphisms, and whose morphisms are $\nf$-independent squares.  (Our assumption on $\nf$ in the preamble of this definition is needed only to ensure that $\ck_\downarrow$ does indeed form a category.)
	\item We say that $\nf$ is \emph{$\lambda$-accessible}, $\lambda$ an infinite regular cardinal, if the category $\ck_\downarrow$ is $\lambda$-accessible.  We say that $\nf$ is \emph{accessible} if it is $\lambda$-accessible for some $\lambda$.
	\item We say that $\nf$ is a \emph{stable independence notion} it if is weakly stable and accessible.
	\end{enumerate}
\end{defin}

We note that in the model-theoretic context, weakenings of stability---for example, simplicity---are obtained by weakening existence and/or uniqueness, these being the more difficult property to verify, model-theoretically.  Here accessibility (which corresponds roughly to what model theorists might refer to as \emph{local and finite character}) is far thornier: by and large, it is easier to detect when a weakly stable independence notion satisfies the (weaker) condition of $\lambda$-continuity:

\begin{defin}\label{defcts}Let $\nf$ be an independence notion on category $\ck$, satisfying the existence and transitivity properties.  We say that $\nf$ is $\lambda$-continuous if $\ck_\downarrow$ is closed in $\ck^2$ under $\lambda$-directed colimits.\end{defin}

That $\lambda$-continuity follows from $\lambda$-accessibility, incidentally, is \cite[3.26]{indep-categ-advances}.

We now turn to the central result of this section, which ensures that if an accessible category $\ck$ with all morphisms monomorphisms has a continuous, weakly stable independence notion and a stable independence notion on a sufficiently nice subcategory $\cl$, then in fact there is a stable independence notion on $\ck$ itself.

In this context, ``sufficiently nice'' will mean, precisely, that $\cl$ is a \emph{cofinal} subcategory of $\ck$:

\begin{defin}\label{defcofin}
	We say that a functor $F: \cl \to \ck$ is \emph{cofinal} if for any $K \in \ck$, and any finite sequence $\seq{F L_i \xrightarrow{f_i} K}{i \in I}$, there exists $L \in \cl$, $K \xrightarrow{g} F L$ and $\seq{F L_i \xrightarrow{F g_i} F L}{i \in I}$ such that $F g_i = g\circ f_i$ for all $i \in I$.
  
  We say that a subcategory $\cl$ of a category $\ck$ is \emph{cofinal} if the inclusion is cofinal.
\end{defin}

\begin{example}\label{cofinal-ex} \
  \begin{enumerate}
  \item If $\ck_\ast$ is a full subcategory of $\ck$ and for every $M \in \ck$ there exists $M \to N$ with $N \in \ck_\ast$, then $\ck_\ast$ is a cofinal subcategory of $\ck$. Of course, by choosing $I=\emptyset$, every cofinal subcategory has the latter property.
  \item The category of $\lambda$-saturated models of an elementary class is a cofinal subcategory. This generalizes to $\mu$-AECs, \cite[7.7]{indep-categ-advances}.
  \end{enumerate}
\end{example}

\begin{remark}\label{remscof}
	We note that this notion of cofinality is somewhat weaker than the one found elsewhere in the category-theoretic literature, e.g. \cite[0.11]{adamek-rosicky}.  One can show, however, that the two definitions coincide if one assumes the amalgamation property.
\end{remark}

\begin{theo}\label{cofinal-thm}
Let $\ck$ be an accessible category and let $\cl$ be an accessibly embedded cofinal full subcategory of $\ck$. If:

  \begin{enumerate}
  \item All morphisms of $\ck$ are monos.
  \item $\ck$ has an $\aleph_0$-continuous weakly stable independence notion.
  \item $\cl$ has a stable independence notion.
  \end{enumerate}

  Then $\ck$ has a stable independence notion.
\end{theo}
\begin{proof}
Let $\nf$ be an $\aleph_0$-continuous weakly stable independence notion on $\ck$. The $\aleph_0$-continuity of $\nf$ implies in particular that $\ck$ has directed colimits, so $\cl$ has directed bounds. It is easy to check that $\nf$ restricted to $\cl$ is a weakly stable independence notion, hence by the canonicity theorem (\cite[A.6]{lrv-cell}) it must be stable. We wish to show that $\nf$ is stable, so must show, in particular, that $\ck_{\downarrow}$ is accessible. Since $\nf$ restricted to $\cl$ is accessible, $\cl$ itself must be an accessible category. Taking $\mu$ bigger if needed, we may assume without loss of generality that $\ck$, $\cl$, and $\cl_{\downarrow}$ are $\mu$-accessible. Let $\ck_\mu$ denote the full subcategory of $\ck$ consisting of $\mu$-presentable objects. Let $\ck^\ast$ consist of those morphisms in $\ck$ which are $\mu$-directed colimits in $\ck_\downarrow$ of morphisms in $\ck_\mu^2$. Recalling that the morphisms in $\ck_\mu^2$ are precisely the $\mu$-presentable objects of $\ck^2$, it suffices to see that 
$\ck^\ast=\ck^2$. Clearly, $\ck_\mu^2=\ck_\mu^\ast$ and $\cl^2\subseteq\ck^\ast$.

 Concretely, we will prove the following:
\begin{itemize}
  \item[(i)] $\ck^\ast$ is closed under composition.
  \item[(ii)] $\ck^\ast$ is left cancellable in $\ck^2$.
  \item[(iii)] If $M \in \ck$, there exists $M\to N$ in $\ck^\ast$ with $N$ in $\cl$.
  \end{itemize}

To verify that this is sufficient, take $f:M\to N$ in $\ck^2$.  We will show that, under assumptions (i) and (ii), if $M$ satisfies the condition of (iii), then $f$ must be in $\ck^\ast$.  To begin, since $M$ satisfies (iii), there is $g:M\to M'$ in $\ck^\ast$ with $M'$ in $\cl$. Without loss of generality, $f\in\ck^2_\lambda$ and $\ck^2$ is $\lambda$-accessible for some $\lambda>\mu$. Following \cite{rosicky-sat-jsl}, there is a morphism $h:M'\to M^\ast$ where $M^\ast$ is $\lambda$-saturated in $\ck$. Thus there is $t:N\to M^\ast$ such that $tf=hg$. Following \ref{cofinal-ex}(1), there is $p:M^\ast\to L$ with $L$ in $\cl$. Since $ph\in\cl^2\subseteq\ck^\ast$, (i) implies that $phg\in\ck^\ast$. Hence, because $ptf=phg$, (ii) implies that $f\in\ck^\ast$, as desired.  We now prove statements (i)-(iii):\\

{\noindent\bf(i)} Consider $f:K\to L$ and $g:L\to M$ in $\ck^\ast$. Let $f=\colim_i f_i$ and $g=\colim_j g_j$ be $\mu$-directed colimits 
in $\ck_\downarrow$ of morphisms in $\ck_\mu^2$. We may assume that $(k_{ii'},l_{ii'}^1):f_i\to f_{i'}$ and 
$(l_{jj'}^2,m_{jj'}):g_j\to g_{j'}$ are canonical diagrams of $f$ and $g$ with respect to $\ck^2_\lambda$ in $\ck_\downarrow$.
Given $i_0\in I$ and $j_0\in J$, there is $j_1\geq j_0$ in $J$ such that $l_{i_0}^1:L_{i_0}^1\to L$ factorizes through 
$l_{j_1}^2:L_{j_1}^2\to L$. Similarly, there is $i_1\geq i_0$ in $I$ such that $l_{j_1}^2$ factorizes through $l_{i_1}^1$. Continuing this procedure and taking a colimit, we get $L_{i_\omega}^1\cong L_{j_\omega}^2$. Then $f_{i_\omega}$ and $g_{j_\omega}$ are composable
and $g_{i_\omega} f_{j_\omega}$ is in $\ck^\ast$. Continuing in this way, we obtain that $gf\in\ck^\ast$.\\

{\noindent\bf(ii)} Consider $f:K\to L$ and $g:L\to M$ such that $gf\in\ck^\ast$. Let $\ck^{\to\to}$ be the category of composable pairs of morphisms in $\ck$ and $(f,g)=\colim_i (f_i,g_i)$ be a $\mu$-directed colimit in $\ck^{\to\to}$ of morphisms in $\ck_\mu^{\to\to}$. Let 
$gf=\colim_j h_j$ be a $\mu$-directed colimit in $\ck_\downarrow$ of morphisms in $\ck_\mu^2$. Here, $f_i:K^1_i\to L_i$,
$g_i:L_i\to M^1_i$ and $h_j:K^2_j\to M^2_j$. As in the preceding argument, for every $i_0\in I$ and $j_0\in J$ there are $i_\omega\geq i_0$ 
in $I$ and $j_\omega\geq j_0$ in $J$ such that $K^1_{i_\omega}\cong K^2_{j_\omega}$. In the same way, for every $i_0\in I$ 
and $j_0\in J$ there are $i_\omega\geq i_0$ in $I$ and $j_\omega\geq j_0$ in $J$ such that $M^1_{i_\omega}\cong M^2_{j_\omega}$.
By iterating both procedures, we show that  for every $i_0\in I$ and $j_0\in J$ there are $i_\omega\geq i_0$ in $I$ and $j_\omega\geq j_0$ in $J$ such that $M^1_{i_\omega}\cong M^2_{j_\omega}$ and $K^1_{i_\omega}\cong K^2_{j_\omega}$. Then 
$h_{j_\omega}=g_{i_\omega}f_{i_\omega}$, which implies that $f_{i_\omega}\to f$ is a morphism in $\ck_\downarrow$. Hence 
$f\in\ck^\ast$.\\

{\noindent\bf(iii)} Assume that the claim does not hold and let $M$ have the smallest presentation rank $r$ among objects violating (iii). Following
\cite[4.2]{beke-rosicky}, $r=\lambda^+\geq\mu^+$. Under the hypothesis of the theorem, $\ck$ is well $\lambda^+$-filtrable (see \cite[8.8(2)]{internal-improved-v3}, noting that, since the morphisms in $\ck$ are monos, filtrability and well-filtrability coincide).  In particular, $M=\colim M_i$ can be expressed as the colimit of a smooth chain
of $\lambda$-presentable objects where $i\leq\lambda^+$. There is $h_0:M_0\to N_0$ in $\ck^\ast$ with $N_0\in\cl$. There is a 
$\nf$-independent square
$$
 \xymatrix@=3pc{
        M_1 \ar@{}\ar[r]^{h_1} & N_1 \\
        M_0 \ar [u]^{m_{01}} \ar [r]_{h_0} &
        N_0 \ar[u]_{n_{01}}
      }
      $$
in $\ck$ and, since $\cl$ is cofinal in $\ck$, we may assume that $N_1$ is in $\cl$.
Since $M_1$ satisfies (iii), $h_1$ is in $\ck^\ast$. We iterate this procedure, proceeding as above at successor stages.  At limit stages, we take colimits, although this requires somewhat more care: for short chains, we must make use of the cofinality of $\cl$ to ensure that the object in the upper right corner is still in $\cl$.  Taking the colimit of the resulting $\lambda^+$-chain, we have $h:M\to N$ in $\ck^{\ast}$ with $N$ in $\cl$ (thanks to the $\lambda^+$-accessibility of $\cl$), which contradicts our initial assumption.
 
\end{proof}

\section{Stable independence in categories of groups and modules}\label{secstab}

That we are able to lift stable independence in the sense of Theorem~\ref{cofinal-thm} yields immediate benefits, namely the proof of stable independence in a host of categories that arise naturally in algebra.  This is a consequence not only of the theorem, but the following recent developments:
\begin{enumerate}
\item One of the essential ideas of \cite{LRVweak} is that continuous weakly stable independent relations are abundant and easily detectable in the algebraic context, typically taking the form of \emph{effective squares}.  We briefly recall some of the necessary terminology, as it will also be required in Section~\ref{sechdim}.
\item In \cite{univ-modules-pp-v2} and \cite{MAFuchs}, a large number of algebraic categories are shown to have precisely the model-theoretic properties required to ensure the existence of a cofinal full subcategory equipped with a stable independence relation.
\end{enumerate}
Taken together, this yields a host of algebraic categories with stable independence.

We begin by recalling a few pieces of necessary terminology from \cite{lrv-cell}, which will allow us to give sufficient conditions for the existence of a $\aleph_0$-continuous weakly stable independence relation.

\begin{defin}\label{defnice} Let $\ck$ be a category, and let $\cm$ be a class of morphisms in $\ck$.
\begin{enumerate}
	\item We say that $\cm$ is \emph{almost nice} if it is satisfies the following conditions:
	\begin{enumerate}
	\item $\cm$ is \emph{normal}: it contains all isomorphisms in $\ck$ and is closed under composition.
	\item $\cm$ is \emph{coherent}: whenever $f$ and $g$ are composable morphisms with $gf\in\cm$ and $g\in\cm$, then $f\in\cm$ as well.
	\item $\cm$ is a \emph{coclan}: the pushout of any two morphisms, at least one of which is in $\cm$, exists, and $\cm$ is closed under pushouts.
	\end{enumerate}
	Incidentally, we say that $\cm$ is \emph{nice} if it is also closed under retracts in $\ck^2$. This means that if $(u,v):g\to f$ and $(r,s):f\to g$ are morphisms in $\ck^2$
	such that $(r,s)(u,v)=\id_g$ then $f\in\cm$ implies that $g\in\cm$.
	\item We say that $\cm$ is \emph{$\lambda$-continuous}, $\lambda$ an infinite regular cardinal, if $\ck$ has $\lambda$-directed colimits, and $\cm$ is closed under $\lambda$-directed colimits in $\ck$.
	\item We say that $\cm$ is $\lambda$-accessible if it is $\lambda$-continuous and both
	$\ck$ and $\ck_\cm$ are $\lambda$-accessible. $\cm$ is accessible if it is $\lambda$-accessible for some $\lambda$.
\end{enumerate}
\end{defin}

\begin{notation}Note that if $\cm$ is normal, we can form a subcategory $\ck_\cm$ of $\ck$ whose objects are those of $\ck$ and whose morphisms are precisely those in $\cm$.\end{notation}

We define a natural candidate for an independence relation on $\ck_\cm$ in the form of \emph{$\cm$-effective squares}, following \cite[2.2,2.3]{lrv-cell}. 

\begin{defin}\label{defmeff}
	Let $\ck$ be a category, and let $\cm$ be a class of morphisms in $\ck$.  An \emph{$\cm$-effective square} is a commutative square of $\ck$-morphisms
	$$\xymatrix{ B\ar[r]^h & D \\
	A\ar[u]^f\ar[r]_g & C\ar[u]_k }$$
	such that the pushout $P$ of $f$ and $g$ exists, and the induced morphism $P\to D$ is in $\cm$.
\end{defin}

We note that, in case $\cm$ consists of the regular monomorphisms in $\ck$, $\cm$-effective squares are precisely the \emph{effective unions} of \cite{effective-unions}.

\begin{fact}\label{factmeffwstab}
	If $\ck$ has pushouts and $\cm$ is almost nice, $\cm$-effective squares form a weakly stable independence relation on $\ck_\cm$ \cite[2.7]{lrv-cell}.  If, moreover, $\cm$ is $\lambda$-continuous, this independence relation is $\lambda$-continuous \cite[2.11]{lrv-cell}.
\end{fact}

This will guarantee the existence of $\aleph_0$-continuous, weakly stable independence relations in a number of familiar algebraic categories that have recently been the subject of analyses using the tools of abstract model theory.  The essential fact we require, which allows us to isolate a cofinal, full subcategory equipped with a stable independence relation---and thus apply Theorem~\ref{cofinal-thm}---is fundamentally model-theoretic.  

We recall that in an abstract elementary class $\ck$, the syntactic types familiar from classical model theory are replaced by \emph{Galois} (or \emph{orbital}) \emph{types}: given a model $M\in\ck$, Galois types over $M$ are typically identified with orbits of tuples in a large, strongly homogeneous \emph{monster model} under automorphisms fixing $M$.  A class is said to be \emph{$\lambda$-Galois stable} if there are at most $\lambda$ Galois (1-)types over any $M\in\ck$ of cardinality $\lambda$.  Moreover, Galois types in a class $\ck$ are said to be \emph{$<\aleph_0$-short} if, roughly speaking, the type of any tuple is completely determined by the types of its finite subtuples.  Readers unfamiliar with these properties may wish to treat the following as a black box:  

\begin{lemma}\label{modthlem}
	Let $\ck$ be an \emph{abstract elementary class}.  If 
	\begin{enumerate}
		\item $\ck$ has the amalgamation property,
		\item $\ck$ is Galois-stable, and
		\item types in $\ck$ are $<\aleph_0$-short over models,
	\end{enumerate} 
	then there is a full, cofinal subcategory of $\ck$---consisting of sufficiently saturated models---on which there is a stable independence notion.
\end{lemma} 

\begin{proof}[Proof sketch] In essence, finite shortness puts us in the realm of homogeneous model theory, where the desired result is already  known. While we omit the full argument, we hope that the following outline will be sufficient for the interested reader.  In the process, we will make free use of the technique of \emph{Galois Morleyization} introduced in [Vas1]: in an $<\aleph_0$-short AEC one can identify types of finite sequences over the empty set with finitary quantifier-free formulas (formally, by expanding the language).

Following \cite[\S 3]{HySh}, the assumptions of the lemma yield a relation \emph{$p$ is free over $M$}, for $p$ a Galois type over a model $N \le M$, that satisfies all the properties of stable independence provided $M$ and $N$ are sufficiently saturated (in particular, in the terminology of that paper, $M$ and $N$ must be $a$-saturated, a consequence of $\lambda$-saturation in some sufficiently big $\lambda$, see \cite[1.9.4]{HySh}). Note that the properties of independence verified in that paper are the model-theoretic analogues of the category-theoretic definition we discuss here, but the two definitions are equivalent, \cite[8.14]{lrv-cell}. There is the issue, too, that in \cite{HySh}, $p$ is assumed to be the type of a finite sequence---as we are concerned with types of infinite sequences, we must show that the existence/extension property of \cite{HySh} can be transferred to this context. This can be done relatively easily, making use of the compactness theorem for homogeneous model theory: the complete type $p$ of a sequence of arbitrary length is satisfiable just in case its restrictions to finite subsequences are satisfiable (see \cite[1.1]{HySh}, or, more explicitly, \cite[3.8,3.9]{abv}). Given the type $p$ of a sequence of arbitrary length over $M$, and $\bar{x}$ a finite subsequence, the restriction of $p$ to $\bar{x}$ has a free extension over $N$.  Consider the set of all such free extensions, regarded as quantifier-free formulas (via Galois Morleyization, if necessary). The resulting set is complete and, by the extension property of freeness for types of \emph{finite} sequences, all of its restrictions to finite sets of variables are consistent. By construction, the resulting type is the free extension of $p$ over $N$.
\end{proof}

We now obtain stable independence relations on a wide array of algebraic categories using Theorem~\ref{cofinal-thm}, Fact~\ref{factmeffwstab}, and Lemma~\ref{modthlem}---we note that \cite{mazarmnonelem} constructs stable independence relations in many of the same cases, by more concrete means.

As a template for our approach, consider:

\begin{theo}\label{r-mod-pure}
  For any ring with unit $R$, the category of (left) $R$-modules and pure monomorphisms, $\Rmod_{pure}$, has a stable independence notion.
\end{theo}
\begin{proof}
  By \cite{univ-modules-pp-v2}, $\Rmod_{pure}$ forms an AEC, has amalgamation, is stable, and types are $(<\aleph_0)$-short over models; that is, it satisfies all the hypotheses of Lemma~\ref{modthlem}.  Thus $\ck$ must have a stable independence relation on its sufficiently saturated models, which form a cofinal, full subcategory. By Theorem \ref{cofinal-thm} and Fact~\ref{factmeffwstab}, then, $\Rmod_{pure}$ has a stable independence relation.
\end{proof}

As an aside, in light of Fact~\ref{thmlrvcell} below (originally appearing as \cite[3.1]{lrv-cell}), it follows that:

\begin{cor}
	Pure monomorphisms are cofibrantly generated (generated from a \emph{set} of morphisms by pushouts, tranfinite composition, and retracts) in $\Rmod$.
\end{cor}

Note that this is a special case of \cite[3.13]{lprv-purecofgen-v3}, which holds not just for $R$-modules but arbitrary locally finitely presentable additive categories.  Of greater interest are the other applications of this style of argument.  In particular,

\begin{theo}\label{thmmodcats}
	Let $R$ be an integral domain. The following categories of modules have a stable independence relation:
	\begin{enumerate}
		\item Torsion $R$-modules with pure monomorphisms.
		\item $R$-divisible modules with pure monomorphisms (recall that a module $M$ is $R$-divisible if for any nonzero $m\in M$ and nonzero $r\in R$, there is $n\in M$ with $rn=m$).
	\end{enumerate}
\end{theo}

\begin{proof}
	\begin{enumerate}
	\item The category of torsion $R$-modules and pure monomorphisms satisfies the condition of Lemma~\ref{modthlem}, by \cite[4.2(4)]{MAFuchs} and \cite[4.8(2)]{MAFuchs}	
	\item Similarly to (1), using \cite[4.2(5)]{MAFuchs} in place of \cite[4.2(4)]{MAFuchs}.
	\end{enumerate}
\end{proof}

We obtain stable independence relations on an assortment of familiar categories of groups, as well, again taking advantage of recent model-theoretic results---again, \cite{mazarmnonelem} actually obtains similar results, by different means.

\begin{theo}\label{thmgpcats}
	The following categories of groups all have a stable independence relation.  
	\begin{enumerate}
		\item Abelian groups with monomophisms (respectively, pure monomorphisms).
		\item Torsion-free abelian groups with pure monomorphisms.  Similarly, reduced torsion-free abelian groups with pure monomorphisms.
		\item Abelian $p$-groups with monomorphisms (respectively, pure monomorphisms), $p$ any prime.
		\item Torsion abelian groups with monomorphisms (respectively, pure monomorphisms).
		\item Divisible abelian groups with monomorphisms.
	\end{enumerate}
\end{theo}

\begin{proof} As in the proof of Theorem~\ref{thmmodcats}, we satisfy ourselves with indicating the model-theoretic sources that ensure the category satisfies the hypotheses of Lemma~\ref{modthlem}.
	\begin{enumerate}
		\item \cite{bcg} and \cite[3.12]{mazsup}, and \cite[3.16]{univ-modules-pp-v2} (respectively).
		\item The reduced case is \cite[1.2(3)]{shuniv}; general torsion free groups are addressed in, e.g. \cite[3.14]{univ-modules-pp-v2}.
		\item By \cite[4.8(3)]{MAFuchs} and \cite[3.5]{MAFuchs}.
		\item By \cite[4.8(3)]{MAFuchs} and \cite[4.8(1)]{MAFuchs}, respectively.
		\item By \cite[4.8(3)]{MAFuchs}.
	\end{enumerate}
\end{proof}

We expect that considerably more applications of Theorem~\ref{cofinal-thm} of this form are within easy reach: the forthcoming paper \cite{mazarmnonelem}, for example, provides clear avenues for future work along these lines.   

\section{Higher-dimensional independence}\label{sechdim}

We turn now to a different, and perhaps more natural question: we know that if we have a stable independence relation on a category $\ck$, we obtain a well behaved subcategory $\ck_\downarrow$ of the category of morphisms $\ck^2$ consisting of the independent squares.  Is it the case, too, that there is a stable independence notion on $\ck_\downarrow$---consisting now of commutative \emph{cubes} in $\ck$---and under what conditions?  Is there, in turn, a stable independence relation on these cubes?  

We wish to examine, in short, the existence and behavior of \emph{higher-dimensional} stable independence relations.

It should be noted that this is not an exercise in abstraction: higher-dimensional independence relations have played a significant role in recent advances in model theory.  The idea, introduced by Shelah, is vital in his analysis of the classification theory of $L_{\omega_1,\omega}$ in \cite{she83} and \cite{she83b}, and in his proof of the first-order Main Gap, cf. \cite[Ch. XII]{shelahfobook}.  More recently, these notions have made crucial appearances in a number of categoricity transfer arguments, most notably in \cite{zilbexcellent}, for quasiminimal pregeometry classes, and \cite{multidim-v2}, for abstract elementary classes with amalgamation, assuming weak GCH.  Of particular interest are \emph{excellent} classes, which possess independence notions in all finite dimensions.

To be precise, we propose the following notions of $n$-dimensional stable independence and excellence, adapted to our context---these should specialize to the standard ones in the model-theoretic examples mentioned above.

\begin{defin}\label{defhdimind}
	Let $\ck$ be a category.  For $n\geq 1$, we define an \emph{$n$-dimensional stable independence notion on $\ck$}, $\Gamma$, and its induced category, $\ck^\Gamma$, proceeding by induction on $n$:
	\begin{itemize}
		\item We say that $\Gamma$ is a $1$-dimensional stable independence notion on $\ck$ just in case it is the collection of all morphisms in $\ck$.  In this case, we define $\ck^\Gamma=\ck$.
		\item An $(n+1)$-dimensional stable independence notion on $\ck$ consists of a pair $(\Gamma_n,\Gamma)$, where 
		\begin{itemize}
			\item $\Gamma_n$ is an $n$-dimensional stable independence notion on $\ck$, and
			\item $\Gamma$ is a stable independence notion on $\ck^{\Gamma_n}$, in the sense of Definition~\ref{defstab}(3).
		\end{itemize}
		\item Given an $(n+1)$-dimensional stable independence notion $\Gamma_{n+1}=(\Gamma_n,\Gamma)$ on $\ck$, we define $\ck^{\Gamma_{n+1}}$ to be the category $(\ck^{\Gamma_n})^\Gamma$, whose objects are morphisms of $\ck^{\Gamma_n}$ and whose morphisms are the $\Gamma$-independent squares.  Note that this is precisely $(\ck^{\Gamma_n})_\downarrow$ with $\nf=\Gamma$.
	\end{itemize}
\end{defin}

Note that the stable independence notions considered in Sections~\ref{seclift} and \ref{secstab} are precisely $2$-dimensional independence notions---in the sense above---on the appropriate categories.  The best case scenario is the following:

\begin{defin}
	We say that a category $\ck$ is \emph{excellent} if for all $n\geq 1$, $\ck$ has an $n$-dimensional stable independence notion $\Gamma_n$ so that $\ck^{\Gamma_n}$ has directed colimits.
\end{defin}

As noted in the introduction, excellence is far from the norm in the model-theoretic context, as the existence property will typically fail for sufficiently high-dimensional relations: \cite{gkkamalgfunct} and \cite{gkktypeamalg} develop a comprehensive theory of such obstructions.  Here we restrict ourselves to the setting of \cite{lrv-cell}; that is, with locally presentable ambient category $\ck$ and well behaved class of morphisms $\cm$.  With this added structure, the obstructions disappear: if there is a stable independence relation on $\ck_\cm$, it is excellent.  

All of the difficulty lies in the inductive step: given an $n$-dimensional stable independence notion $\Gamma$ on $\ck$, how do we construct a stable independence notion $\Gamma'$ so that $(\ck^\Gamma)^{\Gamma'}$ has directed colimits?

In fact, we consider a simpler---but entirely equivalent---problem, whose solution, Proposition~\ref{thmindstep} below, should be of independent interest.  In particular, we take advantage of the fact that, in this framework, existence of stable independence notions of a particular dimension is equivalent to cofibrant generation of a suitable family of morphisms, via the central result of \cite{lrv-cell}.  For the sake of completeness, we include that result here, phrased in terms better suited to the current context:

\begin{fact}\label{thmlrvcell} (\cite[3.1]{lrv-cell})
Let $\ck$ be a locally presentable category, and let $\cm$ be a nice, accessible, and $\aleph_0$-continuous class of morphisms in $\ck$.  The following are equivalent:
\begin{enumerate}
	\item $\ck_\cm$ has a stable independence notion.
	\item $\cm$-effective squares form a stable independence notion in $\ck_\cm$.
	\item $\cm$ is cofibrantly generated in $\ck$.
\end{enumerate}
\end{fact}

Note that we have unpacked much of the terminology used in \cite[3.1]{lrv-cell}, for the benefit of the reader.  Recall that a class of morphisms $\cm$ is \emph{cofibrantly generated} if it can be generated from a set---as opposed to a proper class---of morphisms by pushouts, transfinite compositions, and retracts.

\begin{propo}\label{thmindstep}
Let $\ck$ be locally presentable and $\cm$ be nice and $\aleph_0$-continuous in $\ck$. Let $\cm!$ consist of $\cm$-effective morphisms in $\ck^2$. Then $\cm!$ is nice and 
$\aleph_0$-continuous in $\ck^2$.
\end{propo}
\begin{proof} 
\noindent(a) $\cm!$ is normal: Isomorphisms in $\ck^2$ are commutative squares whose horizontal arrows are isomorphisms. Such squares are pushouts, hence $\cm$-effective. The composition of two $\cm$-effective morphisms is $\cm$-effective (see \cite[2.7]{lrv-cell}).

\noindent
(b) $\cm!$ is coherent: see \cite[2.10]{lrv-cell}.

\noindent
(c) $\cm!$ is a coclan: We have to show that pushouts of $\cm$-effective squares are $\cm$-effective. Let
$$
 \xymatrix@=3pc{
        C \ar@{}\ar[r]^{m_2} & D \\
        A \ar [u]^{m_1} \ar [r]_{m_0} &
        B \ar[u]_{m_3}
      }
$$
be an $\cm$-effective square considered as a morphism $(m_0,m_2):m_1\to m_3$ in $\ck^2$.
Let
$$
 \xymatrix@=3pc{
        C \ar@{}\ar[r]^{c} & C' \\
        A \ar [u]^{m_1} \ar [r]_{a} &
        A' \ar[u]_{m'_1}
      }
$$
be a commutative square considered as a morphism $(a,c):m_1\to m'_1$. Consider the pushout
$$
 \xymatrix@=3pc{
        m_3 \ar@{}\ar[r]^{(b,d)} & m_3' \\
        m_1 \ar [u]^{(m_0,m_2)} \ar [r]_{(a,c)} &
        m_1' \ar[u]_{(m'_0,m_2')}
      }
$$
in $\ck^2$. This means that
$$
 \xymatrix@=3pc{
        A' \ar@{}\ar[r]^{m'_0} & B' \\
        A \ar [u]^{a} \ar [r]_{m_0} &
        B \ar[u]_{b}
      }
$$
and
$$
 \xymatrix@=3pc{
        C' \ar@{}\ar[r]^{m'_2} & D' \\
        C \ar [u]^{c} \ar [r]_{m_2} &
        D \ar[u]_{d}
      }
$$
are pushouts and $m'_3$ is the induced morphism, i.e., $m'_3m'_0=m'_2m'_1$ and $m'_3b=dm_3$.
It suffices to show that the square
$$
 \xymatrix@=3pc{
        C' \ar@{}\ar[r]^{m'_2} & D' \\
        A' \ar [u]^{m'_1} \ar [r]_{m'_0} &
        B' \ar[u]_{m'_3}
      }
$$
is $\cm$-effective.

Consider a pushout
$$
 \xymatrix@=3pc{
        C' \ar@{}\ar[r]^{p'_0} & P' \\
        A' \ar [u]^{m'_1} \ar [r]_{m'_0} &
        B' \ar[u]_{p'_1}
      }
$$
We must show that the induced morphism $t':P'\to D'$ is in $\cm$. Returning to the original square, we know that the induced morphism $t:P\to D$ is in $\cm$ where
$$
 \xymatrix@=3pc{
        C \ar@{}\ar[r]^{p_0} & P \\
        A \ar [u]^{m_1} \ar [r]_{m_0} &
        B \ar[u]_{p_1}
      }
$$
is a pushout. For this, it suffices to show that
$$
 \xymatrix@=3pc{
        P' \ar@{}\ar[r]^{t'} & D '\\
        P \ar [u]^{p} \ar [r]_{t} &
        D \ar[u]_{d}
      }
$$
is a pushout where $p:P\to P'$ is the induced morphism; that is, $pp_0=p'_0c$ and $pp_1=p'_1b$.

Consider morphisms $u:P'\to X$ and $v:D\to X$ such that $up=vt$. Then 
$$
up'_0c=upp_0=vtp_0=vm_2.
$$
Thus there exists a unique $q:D'\to X$ such that $qm'_2=up'_0$ and $qd=v$. It remains to show that $qt'=u$. We have 
$$
qt'p'_0=qm'_2=up'_0
$$
and $qt'p'_1=qm'_3$. To finish the proof, we need that $up'_1=qm'_3$ because then $qt'=u$.
We have
$$
qm'_3m'_0=qm'_2m'_1=up'_0m'_1=up'_1m'_0
$$
and
$$
qm'_3b=qdm_3=vm_3=vtp_1=upp_1=up'_1b.
$$
Hence $qm'_3=up'_1$.

(d) $\cm!$ is nice: Since $\cm$ is closed under retracts, $\cm!$ is closed under retracts.

(e) $\cm!$ is $\aleph_0$-continuous: see \cite[2.11]{lrv-cell}.
\end{proof}

\begin{thm}\label{thmstabexc}
Let $\ck$ be a locally presentable category, and let $\cm$ be a nice, accessible, and $\aleph_0$-continuous class of morphisms in $\ck$.  If $\ck_\cm$ has a stable independence notion, it is excellent.
\end{thm}
\begin{proof}
We have the obvious one-dimensional stable independence notion on $\ck_\cm$, with $(\ck_\cm)^{\Gamma_1}=\ck_\cm$ and $\Gamma_1$ consisting precisely of $\cm$, the class of morphisms in $\ck_\cm$.
Following Fact~\ref{thmlrvcell}, $\cm$-effective squares form a stable, $\aleph_0$-continuous independence notion $\Gamma$ on $\ck_\cm$.  Take $\Gamma_2=(\Gamma_1,\Gamma)$. Then 
$$
(\ck_\cm)^{\Gamma_2}=(\ck_\cm)_\downarrow\subseteq(\ck_\cm)^2\subseteq\ck^2
$$ 
is an accessible category closed under directed colimits in $\ck^2$.

Assume that we have an $n$-dimensional stable independence notion $\Gamma_n$, $n>1$, on $\ck_\cm$ where  
$$(\ck_\cm)^{\Gamma_n}\subseteq(\ck_\cm)^{2^{n-1}}\subseteq\ck^{2^{n-1}}
$$
is an accessible category closed under directed colimits in $\ck^{2^{n-1}}$ and a nice and $\aleph_0$-continuous class $\cm_{n-1}$ of morphisms in $\ck^{2^{n-1}}$. Following Proposition \ref{thmindstep}, $\cm_n=(\cm_{n-1})!$ is nice and $\aleph_0$-continuous class of morphisms in $\ck^{2^n}$. We have to show that $\cm_n$-effective squares yield a stable independence notion on $(\ck_\cm)^{\Gamma_n}$. Following Fact \ref{factmeffwstab}, $\cm_n$-effective squares form a weakly stable independence notion on $(\ck^{2^n})_{\cm_n}$ which is, moreover, $\aleph_0$-continuous. $(\ck_\cm)^{\Gamma_n}$ is a full subcategory of $(\ck^{2^n})_{\cm_n}$ consisting of those $\cm_n$-squares which are $\cm_n$-effective. Since $\cm_n$ is nice, the proof 
of \cite[2.7]{lrv-cell} yields that $\cm_n$-effective squares yield 
a weakly stable independence notion on $(\ck_\cm)^{\Gamma_n}$. Similarly, the proof of \cite[2.11]{lrv-cell} yields that this weakly stable independence notion is $\aleph_0$-continuous.

It remains to show that the category $(\ck_\cm)^{\Gamma_{n+1}}=((\ck_\cm)^{\Gamma_n})_\downarrow$ is accessible.

We note that an $\cm_n$-square
$$
 \xymatrix@=3pc{
        c \ar@{}\ar[r]^{(f'_0,f'_1)} & d \\
        a \ar [u]^{(g_0,g_1)} \ar [r]_{(f_0,f_1)} &
        b \ar[u]_{(h_0,h_1)}
      }
$$
is in fact a cube of the following form:
$$ 
\xymatrix@=3pc{
        & D_1 \ar[rr]^{d} & & D_2\\
        C_1 \ar[rr]^>>>>>>>>>>{c}\ar[ur]^{f_0'} &
        & C_2 \ar[ur]_{f_1'} & \\
        & B_1\ar'[u][uu]^{h_0}\ar'[r][rr]^b & & B_2\ar[uu]_{h_1} & \\
        A_1\ar[ur]^{f_0}\ar[rr]_a\ar[uu]^{g_0} & & A_2\ar[ur]_{f_1}\ar[uu]_>>>>>>>>>{g_1} &  
      }
$$
where the top, bottom, front and rear squares are all $\cm_{n-1}$-effective. 

Such an $\cm_n$-square is $\cm_n$-effective if and only if the derived square induced by pushing out on the right and left hand sides of the cube is $\cm_{n-1}$-effective.  That is, if the following squares are pushouts
$$
 \xymatrix@=3pc{
        C_1 \ar@{}\ar[r]^{p_0} & P & & C_2 \ar@{}\ar[r]^{p'_0} & P'\\
        A_1 \ar [u]^{g_0} \ar [r]_{f_0} &
        B_1 \ar[u]_{p_1} & & A_2 \ar [u]^{g_1} \ar [r]_{f_1} &
        B_2 \ar[u]_{p'_1}
      }
$$
and $q:P\to D_1$, $q':P'\to D_2$ and $p:P\to P'$ are the obvious induced morphisms, the derived square
$$
 \xymatrix@=3pc{
        P' \ar@{}\ar[r]^{q'} & D_2 \\
        P \ar [u]^{p} \ar [r]_{q} &
        D_1 \ar[u]_{d}
      }
$$
must be $\cm_{n-1}$-effective.  

The category $(\ck_\cm)^{\Gamma_{n+1}}$ thus consists of the full subcategory of $((\ck_\cm)^{\Gamma_n})^2$ on the $\cm_{n}$-effective squares whose derived squares are $\cm_{n-1}$-effective.  In particular, the following is a pullback of categories:
$$
 \xymatrix@=3pc{
        ((\ck_\cm)^{\Gamma_n})^2 \ar@{}\ar[r]^{F} & (\ck_\cm)^{2^{n-1}} \\
        \ck^{\Gamma_{n+1}}  \ar [u]^{\bar{G}} \ar [r]_{\bar{F}} &
        (\ck_\cm)^{\Gamma_n}  \ar[u]_{G}
      }
$$
where $F$ sends an $\cm_n$-square to its derived square and $G$ is the inclusion.  The functor $G$ is transportable because $\cm_{n-1}$-effective squares are closed
under isomorphisms of $\cm_{n-1}$-squares. Following \cite[5.1.1]{makkai-pare}, the pullback above is in fact a \emph{Pullback} (or \emph{bipullback}) of the corresponding categories. Hence $(\ck_\cm)^{\Gamma_{n+1}}$ is an accessible category
(see \cite[5.1.6]{makkai-pare}). 
\end{proof}

\begin{remark}
The category $\Rmod_{pure}$ from Theorem \ref{r-mod-pure} is excellent. Similarly, the category $\Rmod_{emb}$ of $R$-modules and regular monomorphisms. The latter follows from the fact that $\Rmod$ has effective unions, hence regular monomorphisms are cofibrantly generated. In fact, the same is true in any Grothendieck abelian category or Grothendieck topos (\cite{effective-unions}).
\end{remark}
 
\bibliographystyle{alpha}
\bibliography{indep-extensions}

\end{document}